\documentclass{amsart}
\usepackage{amsthm,amsfonts,amsmath,amscd,amssymb,latexsym,epsfig}
\usepackage[title,titletoc]{appendix}
\newcommand{\gl}{{\mathfrak g \mathfrak l}}

\newcommand{\cx}{{\mathbb C}}
\newcommand{\diag}{\operatorname{diag}}

\newcommand{\tr}{\operatorname{tr}}

\newcommand{\rank}{\operatorname{rank}}

\newcommand{\Jac}{\operatorname{Jac}}

\newcommand{\D}{{\partial}}

\newcommand{\supp}{\operatorname{supp}}

\newcommand{\Mat}{\operatorname{Mat}}

\numberwithin{equation}{section}

\newtheorem{theorem}{Theorem}[section]

\newtheorem{lemma}[theorem]{Lemma}

\newtheorem{corollary}[theorem]{Corollary}

\newtheorem{proposition}[theorem]{Proposition}

\theoremstyle{remark}

\newtheorem{remark}[theorem]{Remark}

\newtheorem{definition}[theorem]{Definition}

\newtheorem{example}[theorem]{Example}

\renewcommand{\P}{{\mathbb{P}}}

\newcommand{\oC}{{\mathbb{C}}}

\newcommand{\oP}{{\mathbb{P}}}

\newcommand{\oZ}{{\mathbb{Z}}}

\newcommand{\sA}{{\mathcal{A}}}   

\newcommand{\sE}{{\mathcal{E}}}
\newcommand{\sF}{{\mathcal{F}}}
\newcommand{\sG}{{\mathcal{G}}}   

\newcommand{\sM}{{\mathcal{M}}}   

\newcommand{\sO}{{\mathcal{O}}}

\newcommand{\sR}{{\mathcal{R}}}
\newcommand{\sS}{{\mathcal{S}}}

\newcommand{\sU}{{\mathcal{U}}}

\newcommand{\sW}{{\mathcal{W}}}

\newcommand{\fH}{{\mathfrak{h}}}

\begin{document}

\title[Sheaves, resolutions, and coadjoint orbits of loop groups]{Sheaves on $\oP^1\times \oP^1$, bigraded resolutions,  and coadjoint orbits of loop groups}
\author{Roger Bielawski \& Lorenz Schwachh\"ofer}
\address{School of Mathematics\\
University of Leeds\\Leeds LS2 9JT\\ UK}
\address{Fakult\"at f\"ur Mathematik\\ TU Dortmund\\D-44221 Dortmund\\ Germany}


\begin{abstract} We construct a canonical linear resolution of acyclic $1$-dimensional sheaves on $\oP^1\times \oP^1$ and discuss the resulting natural Poisson structure.
\end{abstract}

\subjclass[2000]{58F07, 14F05, 14H40,  14H70, 22E67}
\maketitle

\thispagestyle{empty}

\section{Introduction}

The goal of this paper is to present a (yet another) variation on a theme developed by several authors, notably Moser, Adams, Harnad, Hurtubise, Previato \cite{Moser}, \cite{AHH1}--\cite{AHP}, and relating integrable systems, rank $r$ perturbations, spectral curves and their Jacobians, and coadjoint orbits of loop groups.
\par
Let us briefly recall that, given  matrices $A,Y,F,G$ of size, respectively, $N\times N$, $r\times r$, $N\times r$, and $r\times N$, one defines a $\gl_r(\cx)$-valued rational map
\begin{equation} Y+G(A-\lambda)^{-1}F,\label{rat1}\end{equation}
i.e. an element of the loop algebra $\widetilde{\gl}(r)^-$, consisting of loops extending holomorphically to the outside of some circle $S^1\subset \cx$. This determines a (shifted) reduced coadjoint orbit in $\widetilde{\gl}(r)^-$ (see Remark \ref{reduced} for a definition). On the other hand, the polynomial \eqref{rat1} also determines (generically) a  curve $S$ and a line bundle $L$ of degree $g+r-1$: the curve is defined as the spectrum of \eqref{rat1}, and $L$ is the dual of the eigenbundle of \eqref{rat1}. This  describes $S$ as an affine curve in $\cx^2$, and the isospectral flows, corresponding to Hamiltonians on the space of rank $r$ perturbations, linearise on the Jacobian of the projective model of $S$. 
\par
In fact, as shown by Adams, Harnad, and Hurtubise \cite{AHH1,AHH2}, it is more convenient to  compactify $S$ inside a Hirzebruch surface $F_d$, $d\geq 1$.  This results in singularities, which may be partially resolved, but it gives a particularly nice description of $\Jac^0(S)$, i.e. of the flow directions.
\par
In this paper, we consider a different compactification of $S$, namely inside  $\oP^1\times \oP^1$ and defined as 
\begin{equation} S=\left\{(z,\lambda)\in \oP^1\times \oP^1;\, \det\begin{pmatrix} Y-z & G\\ F & A-\lambda\end{pmatrix}=0\right\}.\label{bigS}\end{equation}
 This is a very natural thing to do, but we know of only one occurence in the literature:  the paper of Sanguinetti and Woodhouse \cite{SW} (we are grateful to Philip Boalch for this reference). In that paper, in addition to other results, the authors use the above compactification to give a nice picture of the duality phenomenon discussed in \cite{AHH3}. Our application is to another subtlety of the rank $r$ perturbation isospectral flow: the fact that the flow  may leave the set where $\rank F=\rank G=r$, without becoming singular. More precisely, we have:
\begin{theorem} Let $S$ be a smooth curve in $\oP^1\times \oP^1$, defined by \eqref{bigS} and corresponding to a (shifted) rank $r$ perturbation of the matrix $A$ ($r\leq N$). A line bundle $L\in \Jac^{g-r+1}(S)$ corresponds to $(A,Y,F,G)$ with $\rank F=\rank G=r$ if and only if $L$ satisfies:
$$ H^0(S,L(0,-1))= H^1(S,L(0,-1))=0,\enskip H^0(S,L(-1,0))=0, \enskip  H^1(S,L(1,-2))=0.$$
\label{Thm1}\end{theorem}
\par
We are interested in more than line bundles on smooth curves in $\oP^1\times \oP^1$. The above approach generalises to acyclic (i.e. semistable) $1$-dimensional sheaves on $\oP^1\times \oP^1$, with a fixed bigraded Hilbert polynomial. In Sections 2 and 3, we construct a natural linear resolution of such a sheaf,  very much in the spirit of Beauville \cite{Beau2}. This gives us  a linear polynomial matrix $M(z,\lambda)$  (up to certain group action).  If the support of the sheaf is a smooth curve of bidegree  $(r,N)$, then the matrix has size $r\times N$.  As long as the point $(\infty,\infty)$ does not belong to the support of the sheaf, then matrices $M(z,\lambda)$ can be identified with the quadruples $A,Y,F,G$. The space $\sM(k,l)$ of the $(A,Y,F,G)$  has a natural Poisson structure, obtained by identifying it with $\gl_N(\cx)^\ast\oplus \gl_r(\cx)^\ast\oplus T^\ast M_{N\times r}(\cx)$. Thus we obtain a Poisson structure on the quotient of an open subset of $\sM(N,r)$ by $GL_N(\cx)\times GL_r(\cx)$. The (generic) symplectic leaves are known, from \cite{AHP,AHH1}, to be reduced coadjoint orbits of loop groups. Our aim is to describe these symplectic leaves directly in terms of sheaves on $\oP^1\times \oP^1$. We show that they correspond to symplectic leaves of a particular  Mukai-Tyurin-Bottacin Poisson structure \cite{Mu,Tyu,Bo,HL,HM1,HM2} on the moduli space  $M_Q(r,N)$ of simple sheaves on $\oP^1\times \oP^1$ with (bigraded) Hilbert polynomial $Nx+ry$. The surface $Q=\oP^1\times \oP^1$ is an example of a {\em Poisson surface} \cite{Bo}, and consequently, for every choice of a Poisson structure on $Q$,  i.e. a section $s$ of the anticanonical bundle $K_Q^\ast\simeq \sO(2,2)$, one obtains a Poisson structure on $M_Q(r,N)$ as a map 
$$T_{[\sF]}^\ast M_Q(r,N)\simeq {\rm Ext}^1_Q(\sF,\sF\otimes K_Q)\stackrel{\cdot s}{\longrightarrow} {\rm Ext}^1_Q(\sF,\sF)\simeq T_{[\sF]}M_Q(r,N).$$
We show that the (generic) symplectic leaves $\gl_N(\cx)^\ast\oplus \gl_r(\cx)^\ast\oplus T^\ast M_{N\times r}(\cx)$, i.e. reduced coadjoint orbits in $\widetilde{\gl}(r)^-$, are the symplectic leaves of the Mukai-Tyurin-Bottacin structure corresponding to $s(z,\lambda)=1$, i.e. to the anticanonical divisor $2\bigl(\{\infty\}\times \oP^1+\oP^1\times\{\infty\}\bigr)$.

\bigskip

{\it Acknowledgement.} We thank Jacques Hurtubise for comments and Philip Boalch for the reference \cite{SW}.

\section{Acyclic sheaves on $\oP^1\times \oP^1$ and their resolutions}

\begin{definition} Let $X$ be a complex manifold and let $\sF$ be a coherent sheaf on $X$. 
\begin{itemize}
\item[(i)] The {\em support} of $\sF$ is the complex subspace  $\supp \sF$ of $X$ defined as the zero-locus of the annihilator (in $\sO_X$) of $\sF$. The dimension $\dim \sF$ of $\sF$ is the dimension of its support. 
\item[(ii)] $\sF$ is {\em pure}, if $\dim \sE=\dim \sF$ for all non-trivial coherent subsheaves $\sE\subset \sF$.
\item[(iii)] $\sF$ is {\em acyclic} if $H^\ast(\sF)=0$.
\end{itemize}\label{CM}
\end{definition}

\begin{remark} In the case of  $1$-dimensional sheaves on a smooth surface $X$, purity of $\sF$ means that, at  at every point $x\in \supp \sF$, the skyscraper sheaf $\cx_x$ does not embed into $\sF_x$. 
In addition, a  $1$-dimensional sheaf $\sF$ on a smooth surface $X$  is pure if and only if it is {\em reflexive}, i.e. after performing the duality $\sF\mapsto \sE{\it xt}^1_X(\sF,K_X)$ twice, we obtain back $\sF$ (up to isomorphism) (see \cite[\S 1.1]{HL}).\label{pure}
\end{remark}

In the remainder of the paper, {\bf all sheaves are coherent}.

\medskip

We shall now consider sheaves on $\oP^1\times \oP^1$.
 For any $p,q\in \oZ$ we denote by $\sO(p,q)$ the line bundle $\pi_1^\ast \sO(p)\otimes \pi_2^\ast\sO(q)$, where $\pi_1,\pi_2:\oP^1\times \oP^1\to \oP^1$ are the two projections. We shall also denote by $\zeta$ and $\eta$ the two affine coordinates on $\oP^1\times \oP^1$.
\par
Let $\sF$ be a sheaf on $\oP^1\times \oP^1$. Associated to $\sF$ is its {\em bigraded Hilbert polynomial}
\begin{equation} P_{\sF}(x,y)=\sum_{x,y\in \oZ} \chi(\sF(x,y)).\end{equation}
The sheaf $\sF$ is $1$-dimensional if and only if $P_{\sF}$ is linear.

We begin by describing a canonical resolution of acyclic  $1$-dimensional sheaves  on $\oP^1\times \oP^1$.

\begin{theorem}  Let $\sF$ be a 1-dimensional acyclic sheaf on $\oP^1\times \oP^1$. Then  $\sF$ has a linear resolution by locally free sheaves of the form
\begin{equation}0\to \sO(-2,-1)^{\oplus k}\oplus \sO(-1,-2)^{\oplus l}\stackrel{M(\zeta,\eta)}{\longrightarrow}\sO(-1,-1)^{\oplus (k+l)}\to \sF\to 0, \label{F}\end{equation}
for some $k,l\geq 0$.
\par
Conversely, any $\sF$ defined as cokernel of a map $M(\zeta,\eta)$ as above with $\det M(\zeta,\eta)\not\equiv 0$ is acyclic and $1$-dimensional. 
\label{acyclic}\end{theorem}
\begin{remark} Let $\sF$ be a $1$-dimensional acyclic sheaf on $\oP^1\times \oP^1$ with $P_{\sF}(x,y)=lx+ky$. Then $\sF$ is semistable
 with respect to  $\sO(1,1)$.\end{remark}
\begin{remark} This resolution is canonical, but not necessarily {\em minimal}, in the sense of being obtained from the minimal resolution of the bigraded module
$\bigoplus_{i,j\in \oZ}H^0(\sF(i,j))$. 
\end{remark}
\begin{proof} Let $h^0(\sF(0,1))=k$ and $h^0(\sF(1,0))=l$, so that $P_{\sF}=lx+ky$. Let $\sE=\sF(1,1)$, and let ${\bf \Gamma_\ast}(\sE)= \bigoplus_{i,j\in \oZ}H^0(\sE(i,j))$ be the associated bigraded module over the bigraded ring ${\bf S}=\bigoplus_{i,j\in \oZ}H^0(\oP^1\times \oP^1,\sO(i,j))$. Furthermore, let ${\bf \Gamma_\ast}(\sE)|_{\geq 0}= \bigoplus_{i,j\geq 0}H^0(\sE(i,j))$ be its truncation. Owing to \cite[Lemma 6.8]{McL-S}, the sheaf associated to ${\bf \Gamma_\ast}(\sE)|_{\geq 0}$ is again $\sE$. Moreover,  \cite[Theorem 6.9]{McL-S} implies, as $\sE(-1,-1)$ is acyclic, that  the natural map
$$  H^0(\sE)\otimes H^0(\oP^1\times \oP^1,\sO(p,q))\longrightarrow H^0(\sE(p,q))$$
is surjective for any $p,q\geq 0$. Therefore, we have a surjective homomorphism
$$ {\bf S}^{\oplus(k+l)}\rightarrow {\bf \Gamma_\ast}(\sE)|_{\geq 0}\to 0$$
of bigraded ${\bf S}$-modules. Since $\sE$ is of pure dimension $1$, its projective dimension is $1$, and, hence, the above homomorphism extends to a linear free resolution  
$$ 0\to \bigoplus_{i=1}^{k+l}{\bf S}(-p_i,-q_i)\to  \bigoplus_{i=1}^{k+l}{\bf S}\rightarrow {\bf \Gamma_\ast}(\sE)|_{\geq 0}\to 0,$$
where $p_i,q_i\geq 0$ and $p_i+q_i>0$ for each $i$.
The corresponding sheaves on $\oP^1\times \oP^1$ give us a locally free resolution of $\sE$: 
\begin{equation} 0\to  \bigoplus_{i=1}^{k+l}\sO(-p_i,-q_i)\to \bigoplus_{i=1}^{k+l}\sO\to \sE\to 0.\label{F15}\end{equation}
Since  $H^\ast(\sE(-1,-1))=0$, either $p_i=0$ or $q_i=0$ for every $i$. Since $h^0(\sE(-1,0))=k$, we deduce, after tensoring \eqref{F15} with $\sO(-1,0)$, that $\sum p_i=k$. Similarly $\sum q_i=l$. Since $h^1(\sE)=0$, none of the $p_i$ or $q_i$ can be greater than $1$, and so,  all nonzero $p_i$ and all nonzero $q_i$ are equal to $1$. This proves the existence of resolution \eqref{F}. 
\par
Conversely, if $\sF$ admits a resolution of the form \eqref{F}, then it is $1$-dimensional. The long exact cohomology sequence implies that $\sF$ is acyclic.
\end{proof}

Let us write $n=k+l$.
The polynomial matrix $M(\zeta,\eta)$ in \eqref{F15} has size $ n\times n$ and is of the form
\begin{equation} \begin{pmatrix} A_0+A_1\zeta & B_0 + B_1\eta
                 \end{pmatrix},\label{M11}
\end{equation}
with $A_0,A_1\in \Mat_{n,k}(\oC)$, $B_0,B_1\in \Mat_{n,l}(\oC)$. 
Let us denote by $\sA(k,l)$ the space of such matrices with nonzero determinant. The group $GL_{n}(\cx)\times GL_k(\cx)\times GL_l(\cx)$ acts on $\sM(k,l)$ via:
\begin{equation} (g,h_1,h_2).\begin{pmatrix}
                        A(\zeta) & B(\eta)\end{pmatrix} = g\begin{pmatrix}
                        A(\zeta) & B(\eta) \end{pmatrix}\begin{pmatrix} h_1^{-1} & 0\\ 0 & h_2^{-1}\end{pmatrix},
\label{action}\end{equation}
and we can restate Theorem \ref{acyclic} as follows:
\begin{corollary} There exists a natural bijection between 
 \begin{itemize}
\item[(a)] isomorphism classes of $1$-dimensional acyclic sheaves $\sF$ on $\oP^1\times \oP^1$ such that $h^0(\sF(0,1))=k$, $h^0(\sF(1,0))=l$,\\
{\bf and}
\item[(b)] orbits of $GL_{k+l}(\cx)\times GL_k(\cx)\times GL_l(\cx)$ on $\sA(k,l)$. \hfill $\Box$
\end{itemize} \label{acyclic2}
\end{corollary}

For a sheaf define by \eqref{F}, we can describe its support as follows.  As a set, the support of $\sF$ is 
$$ S=\{(\zeta,\eta)\in \oP^1\times \oP^1;\; \det M(\zeta,\eta)=0\}.$$ 
Let us write $\det M(\zeta,\eta)=\prod_{i=1}^s q_i(\zeta,\eta)^{k_i}$, where $q_i$ are irreducible polynomials. We define the {\em minimal polynomial} $p_M(\zeta,\eta)$ of $M$ as
$\prod_{i=1}^s q_i(\zeta,\eta)^{r_i}$, where 
\begin{eqnarray*}r_i & = & \max\{a_ib_i;\; \text{at a generic point, $M(\zeta,\eta)$ has a Jordan block of size $a_i$}\\ & & \text{with eigenvalue $q_i(\zeta,\eta)^{b_i}$}\}.\end{eqnarray*}
Then: 
\begin{proposition} 
The support of $\sF$ is the curve  $\bigl(S,\sO_{\oP^1\times \oP^1}/(p_M)\bigr)$.\hfill$\Box$
\label{S}
\end{proposition} 

\medskip

Let us now fix the support $S$. For simplicity, we shall assume that it is an {\em integral} curve in the linear system $|\sO(k,l)|$ on $\oP^1\times \oP^1$, i.e. $S$ is given by an irreducible polynomial $P(\zeta,\eta)$ of bidegree $(k,l)$, $k,l\geq 1$. This immediately implies that the rank of $\sF$ is constant, i.e. $\sF$ is locally free. Theorem \ref{acyclic} and Corollary \ref{acyclic2} imply
\begin{corollary} Let $P(\zeta,\eta)$ be an irreducible polynomial of bidegree $(k,l)$, and $S=\{(\zeta,\eta);\; P(\zeta,\eta)=0\}$ the corresponding integral curve of genus $g=(k-1)(l-1)$. There exists a canonical biholomorphism
 $$ \Jac^{g-1}(S)-\Theta \simeq \left\{M\in \sA(k,l);\; \det M=P\right\}/GL_{n}(\cx)\times GL_k(\cx)\times GL_l(\cx).$$
\end{corollary}

Similarly, let $\sU_S(r,d)$ be the moduli space of semistable vector bundles (locally free sheaves) on $S$. For $d=r(g-1)$ define the generalised theta divisor $\Theta$ as the set of bundles with nonzero section. Then we have:
\begin{corollary} Let $P(\zeta,\eta)$ be an irreducible polynomial of bidegree $(k,l)$, and $S=\{(\zeta,\eta);\; P(\zeta,\eta)=0\}$ the corresponding integral curve of genus $g=(k-1)(l-1)$. There exists a canonical biholomorphism
 $$ \sU_S(r,r(g-1))-\Theta \simeq \left\{M\in \sA(kr,lr);\; \det M=P^r\right\}/GL_{nr}(\cx)\times GL_{kr}(\cx)\times GL_{lr}(\cx).$$
\end{corollary}

\section{A geometric resolution}

There is a much more geometric way of constructing resolution  \eqref{F}, which works  under mild assumptions on the sheaf $\sF$ (cf. \cite{BSch} for the case of $\sigma$-sheaves).
\begin{definition} Let $\sF$ be a $1$-dimensional sheaf on $\oP^1\times \oP^1$ and let $\pi_1,\pi_2:\oP^1\times \oP^1\to \oP^1$ be the two projections. We say that $\sF$ is {\em  bipure}, if $\sF$ has no nontrivial coherent subsheaves supported on $\{z\}\times \oP^1$ or on $\oP^1\times \{z\}$ for any $z\in \oP^1$.
\label{bipure} \end{definition} 

Let now $\sF$ be an acyclic and bipure sheaf on $\oP^1\times \oP^1$ with Hilbert polynomial $lx+ky$. As in the proof of Theorem \ref{acyclic}, we consider the sheaf $\sE=\sF(1,1)$. Let $D_\zeta$ and $D_\eta$ denote the divisors $\{\zeta\}\times \oP^1$, $\oP^1\times \{\eta\}$.   We set
\begin{equation} V_\zeta= \{s\in H^0(\sE); s|_{D_\zeta}=0\}, \quad  W_\eta=\{s\in H^0(\sE); s|_{D_\eta}=0\}.\label{VW}\end{equation}
For any $\zeta$ and $\eta$, consider the maps
$$  \sE(-1,0)\to \sE, \quad \sE(0,-1)\to \sE,$$
given by multiplication by global non-zero sections of $\sO(1,0)$ and $\sO(0,1)$, vanishing at $\zeta$ and $\eta$, respectively. Since $\sE$ is bipure, these maps are injective, and therefore $V_\zeta\simeq H^0(\sE(-1,0))$, $W_\eta\simeq H^0(\sE(0,-1))$ for any $\zeta,\eta$. In particular $\dim V_\zeta=k$, $\dim W_\eta=l$, for any $\zeta$ and $\eta$. Therefore, $\zeta\mapsto V_\zeta$ and $\eta\mapsto W_\eta$ are subbundles of $H^0(\sE)\otimes \sO$ on $\oP^1$. They are isomorphic to $H^0(\sE(-1,0))\otimes \sO(-1)$, and to $H^0(\sE(0,-1))\otimes \sO(-1)$. The isomorphism is realised explicitly  via the map:
 $H^0(\sE(-1,0))\otimes \sO(-1)\to H^0(\sE)\otimes \sO$,
defined as \begin{equation*}H^0(\sE(-1,0))\otimes \sO(-1)\ni (s,(a,b)) \stackrel{m}{\longmapsto} (b\zeta-a)s\in H^0(\sE)\label{subbundle}\end{equation*} (here $(a,b)\in l$, where $l$ is the fibre of $\sO(-1)$ over $[l]$), and similarly for the subbundle $W$. We now define a vector bundle $U$ on $\oP^1\times \oP^1$, the fibre of which at $\zeta,\eta)$ is $V_\zeta\oplus W_\eta$, i.e.:
$$ U\simeq \left(H^0(\sE(-1,0))\otimes \sO(-1,0)\right)\oplus \left(H^0(\sE(0,-1))\otimes \sO(0,-1)\right).$$
  We obtain an injective map of sheaves $\sU\to H^0(\sE)\otimes \sO$. Let $\sG$ be the cokernel, i.e.
\begin{equation} 0\to \sU\longrightarrow H^0(\sE)\otimes \sO\longrightarrow \sG\to 0.\label{U}\end{equation}
We claim that $\sG\simeq \sE$, and so \eqref{U} is a natural resolution of $\sE$. To prove this, tensor the resolution \eqref{F} by $\sO(1,1)$ to obtain:
\begin{equation}0\to \sO(-1,0)^{\oplus k}\oplus \sO(0,-1)^{\oplus l}\stackrel{M(\zeta,\eta)}{\longrightarrow}\sO^{\oplus (k+l)}\to \sE\to 0. \label{FE}\end{equation}
Clearly, the middle term is identified with $H^0(\sE)\otimes\sO$. For any $\zeta_0$, consider the image  of $M(\zeta_0,\eta)$ restricted to $\sO(-1,0)^{\oplus k}|_{\zeta_0}\oplus 0$. This image does not depend on $\eta$, and 
since $\sF$ is bipure, it is exactly $V_{\zeta_0}$, defined in \eqref{VW}, i.e. sections vanishing on $\zeta_0\times \oP^1$. Similarly, for any $\eta_0$, the image of $M(\zeta,\eta_0)$ restricted to $0\oplus \sO(0,-1)^{\oplus l}|_{\eta_0}$ is precisely $W_{\eta_0}$. Hence, there are canonical isomorphisms between both first and second terms in resolutions \eqref{U} and \eqref{FE}, which commute with the horizontal maps. Therefore $\sG\simeq \sE$.

\section{Poisson structure and orbits of loop groups}

According to Corollary \ref{acyclic2}, acyclic sheaves with Hilbert polynomial $lx+ky$ correspond to orbits of $GL_{k+l}(\cx)\times GL_k(\cx)\times GL_l(\cx)$ on $\sA(k,l)$, where $\sA(k,l)$ is the set of polynomial matrices defined in \eqref{M11} and the action is given in \eqref{action}. 
\par
We now make the following assumption about the sheaf $\sF$:
\begin{equation} (\infty,\infty)\not\in \supp \sF.\label{infty}\end{equation}
This can be, of course, always achieved via an automorphism of $\oP^1\times \oP^1$. In terms of the matrix $M(\zeta,\eta)$ corresponding to $\sF$, \eqref{infty} means that $\det(A_1, B_1)\neq 0$. We can, therefore, use the action of $GL_{k+l}(\cx)$ to make $(A_1,B_1)$ equal to minus the identity matrix, so that $M(\zeta,\eta)$ becomes
\begin{equation}
\begin{pmatrix} X-\zeta & F \\ G & Y-\eta \end{pmatrix}, \enskip X\in \Mat_{k,k}(\cx),\,Y\in \Mat_{l,l}(\cx),\, G,F^T\in \Mat_{l,k}(\cx).
\label{M12}\end{equation}
The residual group action is that of conjugation by the block-diagonal $GL_k(\cx)\times GL_l(\cx)$. We denote this group by $K$. 
\begin{remark} We are, essentially, in the situation of \cite{AHP}. The only difference is that we do not fix $X$ or $Y$.
\end{remark}
We denote by $\sM(k,l)$ the space of all matrices of the form \eqref{M12}, which we identify with quadruples $(X,Y,F,G)$ as above. The action of $K=GL_k(\cx)\times GL_l(\cx)$ on $\sM(k,l)$ is given by
\begin{equation} (g,h).(X,Y,F,G)=(gXg^{-1}, hYh^{-1}, gFh^{-1}, hGg^{-1}).\label{action2}\end{equation}
Let us also write $\sS(k,l)$ for the set of isomorphism classes of acyclic sheaves with Hilbert polynomial $lx+ky$ on $\oP^1\times \oP^1$, which satisfy \eqref{infty}. The content of Corollary \ref{acyclic2} is that there exists a natural bijection
\begin{equation}
 \sM(k,l)/K\simeq \sS(k,l).\label{biject}
\end{equation}

\subsection{Poisson structure\label{Poisson}}
The vector space $\Mat_{k,l}\times \Mat_{l,k}$ has a natural $K$-invariant symplectic structure: $\omega=\tr(dF\wedge dG)$. On the other hand, $\Mat_{k,k}\simeq \gl_k(\cx)^\ast$ and $\Mat_{l,l}\simeq \gl_l(\cx)^\ast$ have canonical Poisson structures, and therefore, $\sM(k,l)$ has a natural $K$-invariant Poisson structure. If $\sM(k,l)^0$ is the subset of $\sM(k,l)$, on which the action of $K$ is free and proper, then $\sM(k,l)^0/K$ is a Poisson manifold, and, consequently, we obtain a Poisson structure on the corresponding subset of acyclic sheaves with Hilbert polynomial $lx+ky$ and satisfying \eqref{infty}. We shall now want to describe symplectic leaves of $\sM(k,l)^0/K$ in terms of sheaves on $\oP^1\times \oP^1$.
\par
First of all, let us describe sheaves corresponding to symplectic leaves in $\sM(k,l)$. Such a leaf is determined by fixing conjugacy classes of $X$ and $Y$. On the other hand, conjugacy classes of $k\times k$ matrices correspond to isomorphism classes of torsion sheaves on $\oP^1$, of length $k$. This correspondence is given by associating to a matrix $X\in \Mat_{k,k}(\cx)$ the sheaf $\sG$ via
\begin{equation} 0\to \sO(-1)^{\oplus k}\stackrel{X-\zeta}{\longrightarrow} \sO^{\oplus k}\to \sG\to 0.\label{torsion}\end{equation}
If, for example, $X$ is diagonalisable with distinct eigenvalues $\zeta_1,\dots,\zeta_r$ of multiplicities $k_1,\dots, k_r$, then $\sG\simeq \bigoplus_{i=1}^r \cx^{k_i}|_{\zeta_i}$, i.e. $\sG|_{\zeta_i}$  is the skyscraper sheaf of rank $k_i$.

\begin{proposition} Let $P$ be a conjugacy class of $k\times k$ matrices.
 The bijection \eqref{biject} induces a bijection between
\begin{itemize}
 \item[(i)] orbits of $GL_k(\cx)\times GL_l(\cx)$ on $\{(X,Y,F,G)\in \sM{(k,l)};\: X\in P\}$, and
\item[(ii)] isomorphism classes of sheaves $\sF$ in $\sS(k,l)$ such that $\sF|_{\eta=\infty}$ is isomorphic to $\sG$ defined by \eqref{torsion}.
\end{itemize}\label{infinity}
\end{proposition}
\begin{proof}
 At $\eta=\infty$, the matrix \eqref{M12} becomes $\begin{pmatrix} X-\zeta & 0\\ G & -1     \end{pmatrix}$. The statement follows from \eqref{torsion} and \eqref{F}.
\end{proof}

Therefore symplectic leaves on $\sM(k,l)$ correspond to fixing isomorphism classes of $\sF|_{\eta=\infty}$ and of  $\sF|_{\zeta=\infty}$. Symplectic leaves on $\sM(k,l)^0/K$ are of course smaller than $K$-orbits of symplectic leaves on $\sM(k,l)^0$. They are obtained by fixing $X$ and $Y$ and taking the symplectic quotient of $\Mat_{k,l}\times \Mat_{l,k}$ by ${\rm Stab}(X)\times {\rm Stab}(Y)$. We shall describe sheaves corresponding to a particular symplectic leaf in the case when $X$ and $Y$ are diagonalisable.

\subsection{Orbits of $GL_k(\cx)$ and matrix-valued rational maps\label{rat}}
We consider now only the action of $GL_k(\cx)\simeq GL_k(\cx)\times \{1\}\subset K$ on $\sM(k,l)$. We fix a semisimple conjugacy class of $X$, i.e. we suppose that $X$ is diagonalisable, with distinct eigenvalues $\zeta_1,\dots,\zeta_r$ of multiplicities $k_1,\dots,k_r$. The stabiliser of $X$ is then isomorphic to $\prod_{i=1}^r GL_{k_i}(\cx)$. If the action of $GL_k(\cx)$ is to be free, we must have $k_i\leq l$, $i=1,\dots,r$. Let us diagonalise $X$, so that $X$ has the block-diagonal form $(\zeta_1 \cdot 1_{k_1\times k_1},\dots, \zeta_r \cdot 1_{k_r\times k_r})$, and let $F_i,G_i$ denote the $k_i\times l$ and $l\times k_i$  submatrices of $F,G$ such that  rows of $F$  and the columns of $G$ have the same coordinates as the block $\zeta_i \cdot 1_{k_i\times k_i}$. The action of $GL_k(\cx)$ is free and proper at $(X,Y,F,G)$ if and only if  $\rank F_i=\rank G_i=k_i$ for $i=1,\dots,r$.
\par
As in \cite{AHP, AHH1}, we can associate to each element of $\sM(k,l)$ a $\Mat_{l,l}(\cx)$-valued rational map:
\begin{equation} 
 R(\zeta)= Y+G(\zeta-X)^{-1}F.\label{RR}
\end{equation}
The mapping $(X,Y,F,G)\mapsto R(\zeta)$ is clearly $GL_k(\cx)$-invariant. If $X$ is diagonalisable, as above, i.e. $X=(\zeta_1 \cdot 1_{k_1\times k_1},\dots, \zeta_r \cdot 1_{k_r\times k_r})$, then
\begin{equation} R(\zeta)=Y+ \sum_{i=1}^r\frac{G_iF_i}{\zeta-\zeta_i}.\label{RFG}\end{equation}
We clearly have:
\begin{lemma} Let $P$ be a semisimple conjugacy class of $k\times k$ matrices with eigenvalues $\zeta_1,\dots,\zeta_r$ of multiplicities $k_1,\dots,k_r$. 
The map $(X,Y,F,G)\mapsto R(\zeta)$ induces a bijection between 
\begin{itemize}
 \item[(i)] $GL_k(\cx)$-orbits on $\{(X,Y,F,G)\in \sM(k,l)^0;\: X\in P\}$, and 
\item[(ii)]  the set $\sR_l(P)$ of all rational maps of the form
$$ R(\zeta)=Y+ \sum_{i=1}^r\frac{R_{i}}{\zeta-\zeta_i},$$
where $\rank R_i=k_i$. \hfill $\Box$
\end{itemize}\label{ss}\end{lemma}
                                                                                     
\subsection{Orbits of loop groups} A rational map of the form \eqref{RR}  may be viewed as an element of a loop Lie algebra $\widetilde{\gl}(l)^-$, consisting of maps from a circle $S^1$ in $\cx$, containing the points $\zeta_i$ in its interior, which extend holomorphically outside $S^1$ (including $\infty$). The group $\widetilde{GL}(l)^+$, consisting of smooth maps $g:S^1\to GL_l(\cx)$, extending holomorphically to the interior of $S^1$, acts on $\widetilde{\gl}(l)^-$ by pointwise conjugation, followed by projection to $\widetilde{\gl}(l)^-$.   In particular, if all eigenvalues of $X$ are distinct, then the action is 
$$ g(\zeta) . \left(Y+ \sum_{i=1}^r\frac{R_{i}}{\zeta-\zeta_i}\right)=Y+  \sum_{i=1}^r \frac{g(\zeta_i)R_{i}g(\zeta_i)^{-1}}{\zeta-\zeta_i}.$$
Therefore, if we fix conjugacy classes of the $R_i$, we obtain an orbit of $\widetilde{GL}(l)^+$ in $\widetilde{\gl}(l)^-$. 
We shall now consider  quotients of such  orbits by ${\rm Stab}(Y)$ and describe which sheaves correspond to elements of such an orbit. Let us give a name to such quotients:
\begin{definition} The quotient of an orbit of $\widetilde{GL}(l)^+$ in $\widetilde{\gl}(l)^-$ by $GL_l(\cx)$ is called a {\em semi-reduced orbit}.\label{semi}
\end{definition}
\begin{remark} In the literature (see, e.g. \cite{AHH1}--\cite{AHP}) a {\em reduced} orbit is the symplectic quotient of an orbit by $H_Y={\rm Stab}(Y)$. The $GL_l(\cx)$-moment map on $\widetilde{\gl}(l)^-$ is identified with $Y+\sum_{i=1}^r R_i$, so that a reduced orbit is obtained by fixing the value of $a=\pi\bigl(\sum_{i=1}^r R_i\bigr)$, where $\pi$ is the projection $\gl_l(\cx)\to  \gl_l(\cx)/\fH_Y^\perp$ (with $\perp$ is taken with respect to $\tr$), and dividing by ${\rm Stab}(a)\subset {\rm Stab}(Y)$. Therefore, if ${\rm Stab}(Y)$ fixes $a$, then a reduced orbit can be identified with a subset of a semi-reduced orbit.\label{reduced}
\end{remark}
\par
Let us, therefore, fix a semi-reduced orbit of $\widetilde{GL}(l)^+$. We choose $r$ distinct points $\zeta_1,\dots,\zeta_r$ in $\cx$. Furthermore, we choose $r+1$ conjugacy classes $Q_0,Q_1,\dots, Q_r$ of $l\times l$ matrices. This data determines a semi-reduced orbit $\Upsilon=\Upsilon(Q_0,\dots,Q_r)$ of $\widetilde{GL}(l)^+$ defined as
\begin{equation} \Upsilon=\left\{R(\zeta)=Y+ \sum_{i=1}^r\frac{R_{i}}{\zeta-\zeta_i};\; Y\in Q_0,\,\forall_{i\geq 1} R_i\in Q_i\right\}/GL_l(\cx). \label{orbit}\end{equation}
Let
\begin{equation} k_i=\rank Q_i,\enskip i=1,\dots, r,\quad k=\sum_{i=1}^r k_i.\label{kk}\end{equation}
In the notation of Lemma \ref{ss}, $\Upsilon\subset \sR_l(P)$, where $P$ is the semisimple conjugacy class of $k\times k$ matrices
with eigenvalues $\zeta_i$ of multiplicities $k_i$.
\par
Thanks to Proposition \ref{infinity}, the conjugacy class $P$ determines $\sF|_{\eta=\infty}$, which, in the case at hand, is $\bigoplus_{i=1}^r {\cx^{k_i}}|_{(\zeta_i,\infty)}$. Similarly, $Q_0$ determines the isomorphism class of $\sF|_{\zeta= \infty}$. We now discuss the significance of the other conjugacy classes $Q_1,\dots,Q_l$. 
\par
We claim that they determine the isomorphism class of $\sF|_{\eta^2=\infty}$, i.e. of $\sF$ restricted to the first order neighbourhood of $\eta=\infty$.
Indeed, consider again the canonical resolution \eqref{F} of $\sF$ with $M(\zeta,\eta)$ given by \eqref{M12}. Let $\tilde\eta=1/\eta$ be a local coordinate near $\eta=\infty$, so that 
$$M(\zeta,\tilde\eta)=\begin{pmatrix} 
                       X-\zeta & \tilde\eta F\\ G & \tilde\eta Y-1
                      \end{pmatrix}.$$
Using action \eqref{action}, we can multiply $M(\zeta,\tilde\eta)$ on the right by $\begin{pmatrix} 1 & 0 \\ 0 & (1-\tilde\eta Y)^{-1}\end{pmatrix}$. On the scheme $\tilde\eta^2=0$, we have $(1-\tilde\eta Y)^{-1}=1+\tilde\eta Y$, and so $M(\zeta,\tilde\eta)$ becomes (on $\tilde\eta^2=0$):
\begin{equation*}\begin{pmatrix} X-\zeta & \tilde\eta F\\ G & -1
                      \end{pmatrix}.\label{m2}\end{equation*}
To describe $\sF|_{\tilde\eta^2=0}$, it is enough to describe it near each $\zeta_i$, i.e. to describe $\sG_i=\sF|_{U_i\times \{\tilde\eta^2=0\}}$, where $U_i$ is an open neighbourhood of $\zeta_i$ (not containing the other $\zeta_j$). The resolution \eqref{F} of $\sF$ restricted to $U_i\times \{\tilde\eta^2=0\}$ becomes
$$ \begin{CD}0\to \sO(-2,-1)^{\oplus k_i}\oplus \sO(-1,-2)^{\oplus l} @> M_i(\zeta,\tilde\eta)>>
\sO(-1,-1)^{\oplus(k_i+l)}\longrightarrow \sG_i\to 0,\end{CD}
$$ where
$$ M_i(\zeta,\tilde\eta)=\begin{pmatrix} 
                       \zeta_i-\zeta & \tilde\eta F_i\\ G_i & -1
                      \end{pmatrix}.$$
This implies that we have an exact sequence
\begin{equation}\begin{CD} 0 \to \sO(-2,-1)^{\oplus k_i} @> (\zeta_i-\zeta)+ \tilde\eta F_i G_i>>  \sO(-1,0)^{\oplus k_i} @>>> \sG_i \to 0,
  \end{CD}\label{G_i}\end{equation}
on $U_i\times \{\tilde\eta^2=0\}$. Therefore $\sG_i$ is determined by the $GL_{k_i}(\cx)$-conjugacy class of $F_iG_i$, which is the same as the $GL_l(\cx)$-conjugacy class of $G_iF_i$. Lemma \ref{ss} and formula \eqref{RFG} imply that the conjugacy class of $G_iF_i$ is $Q_i$. Thus, the conjugacy classes $Q_1,\dots,Q_r$, which determine the orbit \eqref{orbit}, correspond to the isomorphism class of $\sF|_{\tilde\eta^2=0}$. Observe that the support of $\sG_i$ is given by $\det((\zeta_i-\zeta)+ \tilde\eta F_i G_i)=0$. In other words, the eigenvalues of $F_iG_i$ give $\frac{\zeta-\zeta_i}{\tilde\eta}$ at $(\zeta,\tilde{\eta})=(\zeta_i,0)$, i.e. the first order neighbourhood of $\supp \sF$ at $(\zeta_i,\infty)$.
\par
Summing up, we have:
\begin{theorem} There exists a natural bijection between elements of the semi-reduced rational orbit \eqref{orbit} of $\,\widetilde{GL}(l)^+$ in $\widetilde{\gl}(l)^-$ and isomorphism classes of $1$-dimensional acyclic sheaves $\sF$ on $\oP^1\times \oP^1$ such that
\begin{itemize}
\item[(i)] the Hilbert polynomial of $\sF$ is $P_{\sF}(x,y)=lx+ky$.
\item[(ii)] $(\infty,\infty)\not \in \supp S$, and  $\sF|_{\eta=\infty}\simeq \bigoplus_{i=1}^r {\cx^{k_i}}|_{(\zeta_i,\infty)}$.
\item[(iii)] The isomorphism class of $\sF|_{\zeta=\infty}$ corresponds to $Q_0$, as in Proposition \ref{infinity}.
\item[(iv)] The isomorphism class of $\sF|_{\eta^2=\infty}$ corresponds to conjugacy classes $Q_1,\dots,Q_r$, as described above. \hfill $\Box$
\end{itemize}
\end{theorem}

\begin{remark} A variation of this result is probably well known to the integrable systems community (at least when $\sF$ is a line bundle supported on a smooth curve $S$). We think it useful, however, to state it in this language and in full generality.
\end{remark}

\subsection{Symplectic leaves of $\sM(k,l)^0/K$}
We can finally describe symplectic leaves of $\sS(k,l)$, i.e. sheaves corresponding to a particular symplectic leaf $L$ in $\sM(k,l)/K$, at least in the case when $L\subset \sM(k,l)^0/K$, and $X$ and $Y$ are semisimple. As we already mentioned in \S\ref{Poisson}, a symplectic leaf in $\sM(k,l)^0/K$ is obtained by fixing  $X$ and $Y$, as well as a coadjoint orbit $\Lambda\subset \fH^\ast$ of $H={\rm Stab}(X)\times {\rm Stab}(Y)$. If $\mu: \Mat_{k,l}\times \Mat_{l,k}\to \fH^\ast$ is the moment map for $H$, then the symplectic leaf determined by these data is:
\begin{equation}
 L=\{(X,Y,F,G)\in \sM(k,l)^0;\, \text{$X$ and $Y$ are given},\enskip \mu(F,G)\in \Lambda\}/H.\label{leaf}
\end{equation}
Let $X$ be diagonal, written as in \S\ref{rat}, i.e. $X=(\zeta_1 \cdot 1_{k_1\times k_1},\dots, \zeta_r \cdot 1_{k_r\times k_r})$ and let $F_i,G_i$, $i=1,\dots, r$, be the corresponding submatrices of $F$ and $G$. Then ${\rm Stab}(X)\simeq \prod_{i=1}^r GL_{k_i}(\cx)$, and the moment map is the projection of the $GL_k(\cx)$-moment map, i.e. $(F,G)\mapsto FG$, onto the Lie algebra of ${\rm Stab}(X)$. In other words, the ${\rm Stab}(X)$-moment map can be identified with \cite{AHP}:
\begin{equation}\mu_X(F,G)=(F_1G_1,\dots,F_rG_r).\label{mX}\end{equation}
Similarly, if $Y$ is diagonal with $s$ distinct eigenvalues of multiplicities $l_1,\dots,l_s$, then we obtain $l_i\times k$ and $k\times l_i$ submatrices
$G^i,F^i$. The stabiliser of $Y$ is isomorphic to $ \prod_{i=1}^s GL_{l_i}(\cx)$ and the moment map is
\begin{equation}\mu_Y(F,G)=(G^1F^1,\dots,G^sF^s).\label{mY}\end{equation}
Therefore, an orbit $\Lambda$ corresponds to $r+s$ conjugacy classes $\pi_1,\dots,\pi_r,\rho_1,\dots,\rho_s$ of $k_i\times k_i$ matrices for the $\pi_i$, and $l_j\times l_j$ matrices for the $\rho_j$.
The leaf $L$ will be contained in $\sM(k,l)^0/K$ if and only if each conjugacy class consists of matrices of maximal rank ($k_i$ or $l_j$).
From the discussion in the previous subsection, we immediately obtain:
\begin{proposition}
 Let $L$ be a symplectic leaf of the Poisson manifold $\sM(k,l)^0/K$, defined as in \eqref{leaf} with semisimple $X$ and $Y$. Then the image of $L$ under the bijection \eqref{biject} consists of isomorphism classes of sheaves $\sF$ in $\sS(k,l)$ such that the isomorphism class of $\sF|_{\zeta^2=\infty}$ and of $\sF|_{\eta^2=\infty}$ is fixed (and determined by $L$).\hfill $\Box$\label{LL}
\end{proposition}
Spelling things out, $X$ determines  $\sF|_{\eta=\infty}\simeq \bigoplus_{i=1}^r {\cx^{k_i}}|_{(\zeta_i,\infty)}$, and each $\pi_i$, $i=1,\dots,r$, determines $\sF$ restricted to a neighbourhood of $(\zeta_i,\infty)$ in $\{\eta^2=\infty\}$ via \eqref{G_i}. Similarly, $Y$ and the $\rho_j$ determine $\sF|_{\zeta^2=\infty}$.

\begin{remark} Symplectic leaves of $\sM(k,l)^0/K$ can be also identified with reduced orbits (cf. Definition \ref{reduced}) of $\widetilde{GL}(l)^+$ in $\widetilde{\gl}(l)^-$. Therefore, the last proposition describes sheaves corresponding to a reduced orbit with $Y$ semisimple. Furthermore, if we view $\sM(k,l)^0/K$ as an open subset of the moduli space of semistable sheaves with Hilbert polynomial $lx+ky$, then this map is a symplectomorphism between the Mukai-Tyurin-Bottacin symplectic structure, described in the introduction, and the Kostant-Kirillov form on a reduced orbit of a Lie group. For an open dense set, where $\sF$ is a line bundle on a smooth curve, this follows from results in \cite{AHH2,AHH4}. Since both symplectic structures extend everywhere, they are must be isomorphic everywhere.
\end{remark}

\begin{example}
 If we want $\sF$ to be a line bundle over its support, then we must require that all $k_i$ and all $l_j$ are equal to $1$. A symplectic leaf in $\sM(k,l)^0/K$ is now given by fixing  diagonal matrices $X=\diag(\zeta_1,\dots,\zeta_k)$ and $Y=\diag(\eta_1,\dots,\eta_l)$ with all $\zeta_i$ and all $\eta_j$  distinct,  as well as the diagonal entries of $FG$ and $GF$, and quotienting by the group of $(k+l)\times (k+l)$ diagonal matrices (acting as in  \eqref{action2}). If the diagonal entries of $FG$ are fixed to be  $\alpha_1,\dots, \alpha_k$, and the diagonal entries of $GF$ are $\beta_1,\dots, \beta_l$, then the corresponding subset of $\sS(k,l)$ consists of sheaves $\sF$ supported on a $1$-dimensional scheme $S$ such that
$$ S\cap \{\eta^2=\infty\}=\bigcup_{i=1}^k \left\{\zeta-\zeta_i=\frac{\alpha_i}{\eta}\right\},\quad 
S\cap \{\zeta^2=\infty\}=\bigcup_{j=1}^l \left\{\eta-\eta_j=\frac{\beta_{j}}{\zeta} \right\}
$$
and the rank of $\sF$ restricted to $S\cap \{\eta^2=\infty\}$ and $S\cap \{\eta^2=\infty\}$ is everywhere $1$.
\end{example}

\begin{remark}
 We expect that Proposition \ref{LL} remains true if $X$ or $Y$ are not semisimple.
\end{remark}

\section{Rank $k$ perturbations}

Let us now assume that $k\leq l$. In \cite{AHH1}, the authors consider Hamiltonian flows on a subset $\sM$ of $\sM^0(k,l)/K$, where 
$\rank F=\rank G=k$. It is clear from the previous section that a generic symplectic leaf of $\sM^0(k,l)/K$ is not contained in $\sM$. Therefore a flow may leave $\sM$ without becoming singular. Since such Hamiltonian flows on a particular symplectic leaf can be linearised on the Jacobian of a spectral curve, it is interesting to know which points of the (affine) Jacobian are outside of $\sM$. We are going to give a very satisfactory answer to this, in terms of cohomology of line bundles.

Let us therefore define the following set:
\begin{equation} 
\sM(k,l)^1=\left\{M\in \sM(k,l)\,;\; \rank F=\rank G=k\right\}.\label{Sm} 
\end{equation}
\begin{remark} The manifold of $GL_k(\cx)$-orbits in $\sM(k,l)^1$  with $X=0$ and fixed $Y$, can be identified with the set $\{Y+GF\}$, i.e. with the {\em space of rank $k$ perturbations of the matrix $Y$}, as considered first by Moser \cite{Moser} ($k=2$), and, then by many other authors, in particular  Adams, Harnad, Hurtubise, Previato \cite{AHP,AHH1}.
\end{remark}

We now ask which acyclic sheaves on $\P^1\times \oP^1$ correspond to  orbits of $K=GL_k(\cx)\times GL_l(\cx)$ on $\sM(k,l)^1$.  We have:
\begin{proposition} Let $k\leq l$. The bijection of Corollary \ref{acyclic2} induces a bijection between: \begin{itemize}\item[(i)]  orbits of $GL_k(\cx)\times GL_l(\cx)$ on $\sM(k,l)^1$, and \item[(ii)] isomorphism classes of acyclic sheaves $\sF$ on $\oP^1\times \oP^1$ with Hilbert polynomial $P_{\sF}(x,y)=lx+ky$, which satisfy, in addition, \eqref{infty} and
$$ H^0(\sF(-1,1))=0 \enskip \text{and} \enskip H^1(\sF(1,-1))=0.$$
\end{itemize}
\end{proposition}
\begin{proof} Consider short exact sequences
$$ \begin{CD} 0\to \sO(-1)^{\oplus k} @>(X-\zeta, G)^T>> \sO^{\oplus (k+l)} @>>> \sW_1\to 0,
\end{CD}$$
$$ \begin{CD} 0\to \sO(-1)^{\oplus l} @>(F,Y-\eta)^T>> \sO^{\oplus (k+l)} @>>> \sW_2\to 0.
\end{CD}$$
The condition that $G$ has rank $k$ is equivalent to $\sW_1$ being a vector bundle, isomorphic to $\sO(1)^{\oplus k}\oplus \sO^{\oplus(l-k)}$. This is equivalent to $H^0(\sW_1\otimes \sO(-2))=0$. On the other hand, we claim that the condition that $F$ has rank $k$ is equivalent to $H^1(\sW_2\otimes \sO(-2))=0$. Indeed, any coherent sheaf on $\oP^1$ splits into sum of line bundles $\sO(i)$ and a torsion sheaf \cite{Q}.  Since $\sW_2$ has a resolution as above, we know that all degrees $i$ in the splitting are nonnegative, and $F$ has rank $k$ if and only if all $i$ are strictly positive, which is equivalent to $H^1(\sW_2\otimes \sO(-2))=0$.
\par
We can use the above exact sequences to obtain two further resolutions of $\sE=\sF(1,1)$:
\begin{equation}0 \to \sO(-1,0)^{\oplus k}\to  \pi_2^\ast \sW_2\to \sE\to 0,\label{R1}
\end{equation}
\begin{equation}0 \to\sO(0,-1)^{\oplus l}  \to  \pi_1^\ast \sW_1\to\sE\to 0,\label{R2}
\end{equation}
where the maps between first two terms are given by the embedding in $\sO^{\oplus (k+l)}$ followed by the projection onto the quotients 
$\sW_2$, $\sW_1$. Tensoring \eqref{R1} with $\sO(0,-2)$ shows that $H^1(\sW_2(-2))=0$ if and only if $H^1(\sE(0,-2))=0$, i.e. $H^1(\sF(1,-1)=0$. Similarly, tensoring \eqref{R2} with $\sO(-2,0)$ shows that $H^0(\sW_1(-2))=0$ if and only if $H^0(\sE(-2,0))=0$, i.e. $H^0(\sF(-1,1)=0$.
\end{proof}

\begin{remark} In the case $k=l$,  $H^0(\sE(-2,0))=0$ implies that $\sE(-2,0)$ is acyclic (and similarly, $H^1(\sE(0,-2))=0$ implies that $\sE(0,-2)$ is acyclic). In other words $\sG=\sE(-1,0)$ satisfies $H^\ast(\sG(-1,0))=H^\ast(\sG(0,-1))=0$. Furthermore, the resolution  \eqref{R2} becomes the following resolution of $\sG$:
\begin{equation}0\to \sO(-1,-1)^{\oplus k}\longrightarrow \sO^k\longrightarrow \sG\to 0.\label{R5}
\end{equation}
 In the case when $S=\supp \sG$ is smooth and $\sG$ is a line bundle, the corresponding part of $\Jac^{g+k-1}(S)$ and the resolution \eqref{R5} have been considered by Murray and Singer in \cite{MS2}.
\end{remark}

\end{document}